\documentclass[a4paper,12pt]{article}

\usepackage{amssymb}
\usepackage{tipa}
\usepackage{bbm}

\usepackage[T2A]{fontenc}
\usepackage[cp1251]{inputenc}
\usepackage{amsmath}

\usepackage[dvips]{graphicx}
\usepackage[
russian,
]{babel}

\usepackage{wasysym}

\newtheorem{theorem}{Теорема}[section]
\newtheorem{corollary}[theorem]{Следствие}
\newtheorem{lemma}[theorem]{Лемма}
\newtheorem{proposition}[theorem]{Предложение}
\newtheorem{notation}[theorem]{Обозначение}
\newtheorem{remark}[theorem]{Замечание}

\newtheorem{ques}[theorem]{Вопрос}

\newtheorem{example}[theorem]{Пример}
\newtheorem{definitionhead}[theorem]{Определение}
\newenvironment{definition}{\begin{definitionhead}%
\sl}{\end{definitionhead}}

\newtheorem{proof}{Доказательство.}

\begin{document}

%\title{Субэкспоненциальная оценка  в теореме Ширшова о высоте}

\title{
%\ \hbox to \textwidth{\normalsize      % Dirty trick to shift
%16R+05E\hfill} % УДК to the left
\ \hbox to \textwidth{\normalsize
УДК 512.54+512.57+512.562+519.1\hfill}\\[1ex]
Оценки на количество перестановочно упорядоченных множеств\footnote{работа выполнена при частичной финансовой поддержке Фонда Саймонса}}
\author{
М.~И.~Харитонов}
%Mikhail Kharitonov}

\maketitle

\begin{abstract}
В работе доказывается, что $\varepsilon_k(n)$ -- количество $n$-элементных перестановочно упорядоченных (т. е. размерности 2) множеств с максимальной антицепью длины $k$ не больше, чем $\min\{{k^{2n}\over (k!)^2}, {(n-k+1)^{2n}\over ((n-k)!)^2}\}$. Также доказывается, что $\xi_k(n)$ -- количество перестановок чисел от 1 до $n$ с максимальной убывающей  подпоследовательностью длины не больше $k$ -- не больше, чем ${k^{2n}\over ((k-1)!)^2}.$ Проводится обзор биекций и связей между парами линейных порядков, парами диаграмм Юнга, целочисленными двумерными массивами и целочисленными
матрицами на основе работ \cite{Sch61}, \cite{Kn70}. Указана полученная в работе \cite{Ges90} производящая функция для количества полилинейных слов длины $n$ над $l$-буквенным алфавитом $(n \leqslant l)$, в каждом из которых не найдётся последовательности из $(k+1)$ буквы в
порядке лексикографического убывания.
\end{abstract}

%{\small}
%\medskip


\section{Введение и основные понятия
}

В работе оценивается количество перестановочно упорядоченных множеств.

\begin{definition}
Частично упорядоченное множество $M$ называется {\em перестановочно упорядоченным,} если порядок на нём есть пересечение двух линейных порядков.
\end{definition}

Рассмотрим теперь некоторую перестановку $\pi$ элементов $1, 2,\dots , n$ (иначе говоря, $\pi\in S_n$). Определим понятие $k$-разбиваемости.

\begin{definition}
Пусть для перестановки $\pi\in S_n$ найдётся последовательность натуральных чисел $1\leqslant i_1\leqslant i_2\leqslant\dots\leqslant i_k$ таких, что $\pi(i_1)\geqslant\pi(i_2)\geqslant\dots\geqslant\pi(i_k).$ Тогда перестановка $\pi(1)\pi(2)\dots\pi(n)$ называется $k$-разбиваемой.
\end{definition}

\begin{example}
Количество не 3-разбиваемых перестановок в группе $S_n$ есть $n$-е число Каталана и равно ${(2n)!\over n!(n+1)!}$.
\end{example}

\begin{proposition}
Если слово является $k$-разбиваемым, то для любого $m<k$ оно также является $m$-разбиваемым.
\end{proposition}

Далее нам потребуется определение диаграммы Юнга.

\begin{definition}
{\em (Стандартной) диаграммой Юнга порядка $n$} называется таблица, в ячейках которой написаны $n$ различных натуральных чисел, причём суммы чисел в каждой строке и каждом столбце возрастают, между числами нет пустых ячеек и есть элемент, который содержится и в первой строке, и в первом столбце.
\end{definition}

\begin{definition}
Диаграмма Юнга называется {\em диаграммой формы} $p = (p_1, p_2,\dots , p_m),$ если у неё $m$ строк и $i$-я строка имеет длину $p_i.$
\end{definition}

Формы диаграмм Юнга пробегают все возможные разбиения на циклы элементов симметрической группы $S_n$. Любой класс сопряжённости группы $S_n$ задаётся некоторым разбиением на циклы. Каждому классу сопряжённости группы соответствует некоторое её неприводимое представление. Следовательно, форма диаграммы Юнга соответствует неприводимому представлению группы $S_n$.

Пронумеруем все клетки диаграммы Юнга формы $p$ числами от 1 до $n$. Пусть $h_k$ -- количество клеток диаграммы Юнга, расположенных
\begin{itemize}
    \item либо в одной строке, либо в одном столбце с клеткой с номером $k$,
    \item находящихся не левее или не выше клетки с номером $k$.
\end{itemize}
Тогда число диаграмм Юнга формы $p$ и равная ему размерность соответствующего неприводимого представления группы $S_n$, вычисляются по ``формуле крюков'' ${n!\over\prod\limits_{k=1}^n h_k}$.

В работе \cite{Sch61} приведена биекция между перестановками $\pi$ чисел $1, 2,\dots , n$ и заполненных теми же числами парами диаграмм Юнга $(P, Q)$. Эта биекция и её следствия будут разобраны в главе \ref{sec:main}.

{\it В нашей работе мы доказываем следующие результаты:}

\begin{theorem}\label{th:main}
$\xi_k(n)$ -- количество не $(k+1)$-разбиваемых перестановок $\pi\in S_n$ -- не больше, чем ${k^{2n}\over ((k-1)!)^2}$.
\end{theorem}

\begin{theorem}\label{th:main2}
$\varepsilon_k(n)$ -- количество $n$-элементных перестановочно упорядоченных множеств с максимальной антицепью длины $k$ -- не больше, чем $\min\{{k^{2n}\over (k!)^2}, {(n-k+1)^{2n}\over ((n-k)!)^2}\}$.
\end{theorem}

\begin{corollary}\label{cor:main}
Пусть $\digamma$ является множеством слов алфавита из $l$ букв с введённым на них лексикографическим порядком. Назовём {\em полилинейным} слово, все буквы которого различны.  Назовём слово {\em $k$-разбиваемым,} если в нём найдутся $k$ непересекающихся подслов, идущих в порядке лексикографического убывания. Тогда количество полилинейных слов длины $n$ $(n\leqslant l)$, не являющихся $(k+1)$-разбиваемыми, не больше, чем ${l!k^{2n}\over n!(l-n)!((k-1)!)^2}$.
\end{corollary}

Оценка в теореме \ref{th:main} улучшает полученную в работе \cite{Lat72}. Следует сказать, что оценка на $\xi_k(n)$ в работе \cite{Lat72} была получена для доказательства теоремы Регева, вопрос же о её точности не ставился. Оценка в работе \cite{Lat72} доказывается с помощью теоремы Дилуорса. Применение теоремы Дилуорса в некоторых других задачах комбинаторики слов описано в работе \cite{BK12}.

В работе \cite{Ch07} доказывается, что для для определённой функции $K(n) = o(\sqrt[3]{n}\ln n)$ и числа  $k \leqslant K(n) = o(\sqrt[3]{n}\ln n)$ верна асимптотическая оценка $\xi_k(n) = k^{2n - o(n)}$.

Для получения производящей функции в работе \cite{Kn70} введено следующее понятие:

\begin{definition}
Обобщённой диаграммой Юнга формы $(p_1, p_2,\dots  p_m)$, где $p_1\geqslant p_2\geqslant\dots\geqslant p_m\geqslant 1$, называется массив $Y$ положительных чисел $y_{ij}$, где $1\leqslant j\leqslant p_i$, $1\leqslant i\leqslant m,$ такой, что числа в его строках не убывают, а в столбцах возрастают.
\end{definition}

Ещё требуются двухстрочные массивы следующего типа.

\begin{definition}
Набор пар положительных чисел $(u_1, v_1), (u_2, v_2),\dots , (u_N, v_N)$ такой, что пары $(u_k, v_k)$ расположены в неубывающем лексикографическом порядке, называется {\em набором типа $\alpha(N)$}.
\end{definition}
В работе \cite{Kn70} устанавливается биекция между наборами типа $\alpha(N)$ и парами $(P, Q)$ обобщённых диаграмм Юнга порядка $N$ (т. е. состоящих из $N$ ячеек). Кроме того, существует взаимно-однозначное соответствие между рассматриваемыми наборами и матрицами, в которых число в ячейке из $i$-ой строки и $j$-го столбца равно количеству пар $(i, j)$ в наборе.
В работе \cite{Ges90} на основании функций Шура $s_\lambda$, которые также являются производящими функциями для обобщённых диаграмм Юнга, строится производящая функция для $\xi_k(n).$ Однако сложность построения явной формулы для $\xi_k(n)$ растёт экспоненциально по $k$. К примеру, $\xi_3(n)= 2 \sum\limits_{ k=0}^n \bigl(\begin{smallmatrix}2k\\ k \end{smallmatrix}\bigr)\bigl(\begin{smallmatrix}n\\ k \end{smallmatrix}\bigr)^2 {3k^2 + 2k + 1 - n - 2kn\over (k + 1)^2(k + 2)(n - k + 1)}$.

\section{Алгебраические обобщения}

В 1950 году Шпехт (см. \cite{Sp50}) поставил проблему существования бесконечно базируемого многообразия ассоциативных алгебр над полем характеристики 0. Решение проблемы Шпехта для нематричного случая представлено в докторской диссертации В. Н. Латышева \cite{Lat77}. Рассуждения В. Н. Латышева основывались на применении техники частично упорядоченных множеств. А. Р. Кемер (см. \cite{Kem87} доказал, что каждое многообразие ассоциативных алгебр конечно базируемо, тем самым решив проблему Шпехта.

Первые примеры бесконечно базируемых ассоциативных колец были получены А. Я. Беловым (см. \cite{Bel99}), А. В. Гришиным (см. \cite{Gr99}) и В. В. Щиголевым (см. \cite{Shch99}).

После решения проблемы Шпехта в случае характеристики 0 актуален вопрос, поставленный Латышевым.

Введём некоторый порядок на словах алгебры над полем. Назовём {\it обструкцией} полилинейное слово, которое
\begin{itemize}
    \item является уменьшаемым (т. е. является комбинацией меньших слов);
    \item не имеет уменьшаемых подслов;
    \item не является изотонным образом уменьшаемого слова меньшей длины.
\end{itemize}

\begin{ques}[Латышев]
Верно ли, что количество обструкций для полилинейного $T$-идеала конечно?
\end{ques}

Из проблемы Латышева вытекает полилинейный случай проблемы конечной базируемости для алгебр над полем конечной характеристики. Наиболее важной обструкцией является обструкция $x_n x_{ n-1}\dots x_1$, её изотонные образы составляют множество не $n$-разбиваемых слов.

В связи с этими вопросами возникает проблема:

\begin{ques}
Перечислить количество полилинейных слов, отвечающих данному конечному набору обструкций. Доказать элементарность соответствующей производящей функции.
\end{ques}

\section{Доказательство основных результатов}\label{sec:main}

\begin{lemma}[\cite{Sch61}]\label{lem:schlem3}
Существует взаимооднозначное соответствие между перестановками $\pi\in S_n$ и парами $(P, Q)$ стандартных диаграмм Юнга, заполненных числами от 1 до $n$ и такими, что форма $P$ совпадает с формой $Q$.
\end{lemma}

\begin{proof}
Пусть $\pi = x_1 x_2\dots x_n$. Построим по ней пару диаграмм Юнга $(P, Q)$. Сначала построим диаграмму $P$.

Определим операцию $S\leftarrow x$, где $S$ -- диаграмма Юнга, $x$ -- натуральное число, не равное ни одному из чисел в диаграмме $S$.
\begin{enumerate}
\item Если $x$ не меньше самого правого числа в первой строке $S$ (если в ней нет чисел, то будем считать, что $x$ больше любого из них), то добавляем $x$ в конец первой строки диаграммы $S$. Полученная диаграмма $S\leftarrow x$.
\item Если найдётся большее, чем $x$, число в первой строке $S$, то пусть $y$ -- наименьшее число в первой строке, такое что $y > x$. Тогда заменим $y$ на $x$. Далее проводим с $y$ и второй строкой те же действия, что проводили с $x$ и первой строкой.
\item Продолжаем этот процесс строка за строкой, пока какое-нибудь число не будет добавлено в конец строки.
\end{enumerate}
Из построения $S\leftarrow x$ получаем, что вновь полученная таблица будет диаграммой Юнга.

Пусть $P = (\dots ((x_1\leftarrow x_2)\leftarrow x_3)\dots \leftarrow x_n).$ Тогда $P$ является диаграммой Юнга и соответствует перестановке $\pi$. Пусть диаграмма $Q$ получается из диаграммы $P$ путём замены $x_i$ на $i$ для всех $i$ от 1 до $n$. Тогда $Q$ также является диаграммой Юнга.

Далее в работе \cite{Sch61} показывается, что приведённое построение пар диаграмм Юнга $(P, Q)$ по перестановкам $\pi\in S_n$ взаимнооднозначно.
\end{proof}

Из алгоритма, приведённого в доказательстве леммы \ref{lem:schlem3} следует

\begin{lemma}[\cite{Sch61}]\label{lem:schth1}
Количество строк в диаграмме $P$ равно длине максимальной убывающей подпоследовательности символов в $\pi = x_1 x_2\dots x_n$.
\end{lemma}

Приступим теперь непосредственно к доказательству теоремы \ref{th:main}.

Рассмотрим перестановку $\pi = x_1 x_2\dots x_n$. Она не $(k+1)$-разбиваема тогда и только тогда, когда в соответствующих ей диаграммах $P$ и $Q$ не больше $k$ строк.

Покрасим числа от 1 до $n$ в $k$ цветов произвольным образом. Таких раскрасок $k^n$. Рассмотрим теперь таблицы (не Юнга!), построенные следующим образом. Теперь для каждого $i$ от 1 до $k$ поместим в $i$-ю строку строимой таблицы числа $i$-го цвета в возрастающем порядке так, чтобы наименьшее число в строке стояло в первом столбце и между числами в одной строке не было пустых ячеек (но целиком пустые строки быть могут). Назовём полученные таблицы {\it таблицами типа $\beta (n, k)$}. Между раскрасками в $k$ цветов чисел от 1 до $n$ и таблицами типа $\beta(n, k)$ есть естественная биекция, следовательно, таблиц типа $\beta(n, k)$ будет ровно $k^n$. Заметим, что любая диаграмма Юнга, заполненная числами от 1 до $n$ с не более, чем $k$ строками, будет таблицей типа $\beta(n, k)$. Будем считать, что таблицы $A$ и $B$ типа $\beta(n, k)$ эквивалентны ($A\sim_\beta B$), если одну из другой можно получить при помощи перестановки строк. Тогда если в таблице типа $\beta(n, k)$ не больше одной пустой строки, то в соответствующем классе эквивалентности будет ровно $k!$ элементов. Так как в диаграммах Юнга числа в столбцах строго упорядочены по возрастанию, то в каждом классе эквивалентности таблиц типа $\beta(n, k)$ будет не более одной диаграммы Юнга. Если в диаграмме Юнга ровно $k$ строк, то в соответствующей таблице типа $\beta(n, k)$ не будет пустых строк. Следовательно, диаграмм Юнга, заполненных числами от 1 до $n$ и имеющих ровно $k$ строк, не больше, чем ${k^n\over k!}$.

Если в диаграмме Юнга $k$ строк, то в ней не больше, чем $(n-k+1)$ столбец. Раскрасим числа от 1 до $n$ в $(n-k+1)$ цвет. Рассмотрим теперь таблицы (не Юнга!), построенные следующим образом. Теперь для каждого $i$ от 1 до $(n-k+1)$ поместим в $i$-й столбец строимой таблицы числа $i$-го цвета в возрастающем порядке так, чтобы наименьшее число в столбце стояло в первой строке и между числами в одном столбце не было пустых ячеек (но целиком пустые стобцы быть могут). Назовём полученные таблицы {\it таблицами типа $\gamma (n, k)$}. Между раскрасками в $(n-k+1)$ цветов чисел от 1 до $n$ и таблицами типа $\gamma (n, k)$ есть естественная биекция, следовательно, таблиц типа $\gamma (n, k)$ будет ровно $k^n$. Заметим, что любая диаграмма Юнга, заполненная числами от 1 до $n$ с $k$ строками, будет таблицей типа $\gamma (n, k)$. Будем считать, что таблицы $A$ и $B$ типа $\gamma (n, k)$ эквивалентны ($A\sim_\gamma B$), если одну из другой можно получить при помощи перестановки столбцов. Пусть в таблице $A$ ровно $t$ ненулевых столбцов. Всего таблиц типа $\gamma (n, k)$ с $t$ ненулевыми строками будет не более, чем таблиц типа $\gamma (n, n-t+1)$, т. е. не более $t^n$. В классе эквивалентности таблицы типа $\gamma (n, k)$ с $t$ непустых столбцов будет $(\min\{t + 1, n-k+1\})!$ элементов. При этом таблиц с $(n-k)$ или $(n-k+1)$ столбцов будет не более $(n-k+1)^n$ и в каждом классе эквивалентности среди них будет $(n-k+1)!$ элементов. Так как в диаграммах Юнга числа в строках строго упорядочены по возрастанию, то в каждом классе эквивалентности таблиц типа $\gamma (n, k)$ будет не более одной диаграммы Юнга. Следовательно, среди таблиц типа $\gamma (n, k)$ будет не более ${(n-k+1)^n\over (n-k+1)!} + \sum\limits_{ t=1}^{n-k-1}{t^n\over t!}\leqslant {(n-k+1)^n\over (n-k)!}$ диаграмм Юнга.

Значит, пар диаграмм Юнга, в каждой из которых по $k$ строк, не больше, чем $\min\{{(n-k+1)^{2n}\over ((n-k)!)^2}, {k^{2n}\over (k!)^2}\}$. Следовательно, существует не больше $\min\{{(n-k+1)^{2n}\over ((n-k)!)^2}, {k^{2n}\over (k!)^2}\}$ перестановок $\pi\in S_n$ с длиной максимальной убывающей подпоследовательности ровно $k$.

Каждая перестановка соотвестсвует с точностью до изоморфизма паре линейных порядков из $n$ элементов. Порядок в перестановочно упорядоченном множестве есть пересечение двух линейных порядков. Так как у каждой пары линейных порядков ровно одно их пересечение, то по леммам \ref{lem:schlem3} и \ref{lem:schth1} количество перестановочно упорядоченных множеств порядка $n$ с максимальной антицепью длины $k$  не больше, чем $\min\{{(n-k+1)^{2n}\over ((n-k)!)^2}, {k^{2n}\over (k!)^2}\}$. Тем самым теорема \ref{th:main2} доказана.

\begin{remark}
Отметим, что по перестановочно упорядоченному множеству не всегда можно определить, какой именно парой линейных порядков оно порождено. Например, рассмотрим множество $\{ p_i \}_{ i=1}^{15}$ с порядком $(p_1>p_2>p_3, p_4>p_5>\dots >p_8, p_9>\dots >p_{15})$. Оно могло быть порождено:
\begin{itemize}
\item парой линейных порядков с соотношениями $(p_3>p_4, p_8>p_9)$ и $(p_3<p_4, p_8<p_9)$,
\item парой линейных порядков с соотношениями $(p_3>p_9, p_15>p_1)$ и $(p_3<p_9, p_15<p_1)$.
\end{itemize}

Эти 2 пары линейных порядков не изоморфны друг другу.
\end{remark}

Оценим $\beth_k(n)$ -- количество диаграмм Юнга, заполненных числами от 1 до $n$ и имеющих не больше $k$ строк.

\begin{lemma}
Верно неравенство $\beth_k(n)\leqslant{k^n\over (k-1)!}$.
\end{lemma}
\begin{proof}
Как показывалось ранее, если в таблице типа $\beta(n, k)$ не больше одной пустой строки, то в соответствующем классе эквивалентности будет ровно $k!$ элементов. Следовательно, диаграмм Юнга, заполненных числами от 1 до $n$ и имеющих либо $(k-1)$, либо $k$ строк, не больше ${k^n\over k!}$. Значит, для $k < 3$ лемма доказана. Пусть она доказана для $k<t$. Тогда для $k=t$ имеем $\beth_k(n)\leqslant {k^n\over k!}+\sum\limits_{i=1}^{ k-2}{i^n\over (i-1)!}\leqslant {k^n\over (k-1)!}.$
\end{proof}

Значит, пар диаграмм Юнга порядка $n$, в каждой из которых по  $\leqslant k$ строк, не больше, чем ${k^{2n}\over ((k-1)!)^2}$. Следовательно, по леммам \ref{lem:schlem3} и \ref{lem:schth1} количество не $(k+1)$-разбиваемых перестановок $\pi\in S_n$ меньше ${k^{2n}\over ((k-1)!)^2}$. Тем самым, теорема \ref{th:main} доказана.

Выведем из теоремы \ref{th:main} следствие \ref{cor:main}. Для каждого набора букв $a_{ i_1}, a_{ i_2},\dots , a_{ i_n}$ количество не $(k+1)$-разбиваемых полилинейных слов длины $n$, составленных из этого набора букв, не больше, чем ${k^{2n}\over ((k-1)!)^2}$. Каждому полилинейному слову отвечает ровно один набор из $n$ букв. Так как наборов из $n$ букв ровно $\bigl(\begin{smallmatrix}l\\ n \end{smallmatrix}\bigr)$, то количество не  $(k+1)$-разбиваемых полилинейных слов длины $n$ не больше, чем ${l!k^{2n}\over n!(l-n)!((k-1)!)^2}.$ Тем самым, следствие \ref{cor:main} доказано.

\section{Обобщенные диаграммы Юнга и их производящие функции}

\begin{lemma}[\cite{Kn70}]\label{lem:knuth}
Существует взаимооднозначное соответствие между наборами типа $\alpha(N)$  и парами $(P, Q)$ обобщённых диаграмм Юнга порядка $N$ у которых форма $P$ совпадает с формой $Q$.
\end{lemma}

\begin{proof}
Определим операцию $S\leftarrow x$, где $S$ -- обобщённая диаграмма Юнга, $x$ -- натуральное число, так же, как в доказательстве леммы \ref{lem:schlem3}. Сопоставим некоторому набору типа $\alpha(N)$ из пар $(u_1, v_1), (u_2, v_2),\dots (u_N, v_N)$ диаграмму Юнга $P = (\dots ((v_1\leftarrow v_2)\leftarrow v_3)\dots \leftarrow v_N)$. Пусть диаграмма $Q$ получается из диаграммы $P$ путём замены $v_i$ на $u_i$ для всех $i$ от 1 до $N$. Тогда $Q$ также является диаграммой Юнга.

Далее в работе \cite{Kn70} показывается, что приведённое построение пар обобщённых диаграмм Юнга $(P, Q)$ по наборам типа $\alpha$ взаимнооднозначно.
\end{proof}

\begin{notation}
Перестановка $\pi\in S_n$ является набором типа $\alpha(n)$ из пар//
 $(1, \pi(1)),\dots ,(n, \pi(n))$.
\end{notation}

{\it Симметрические функции.}

Здесь и далее считаем, что множество индексов при переменных симметрический функций является множеством натуральных чисел.

Напомним несколько понятий из теории симметрических функций.

{\it Полная симметрическая функция} $h_n$ равна $h_n = \sum\limits_{i_1\leqslant i_2\leqslant\dots\leqslant i_n}x_{i_1}x_{i_2}\dots x_{i_n}.$

Пусть $\lambda$ -- набор $(\lambda_1, \lambda_2,\dots ,\lambda_k)$ для некоторого натурального $k$. Пусть также $|\lambda| = \sum\limits_{i=1}^k \lambda_i$. Набор $\lambda$ называется {\it разбиением,} если $\lambda_1\geqslant \lambda_2\geqslant\dots \geqslant\lambda_k.$

{\it Функция Шура} $S_\lambda$ равна $S_\lambda = \det(h_{\lambda_i + j-i})_{1\leqslant i, j\leqslant k}.$

Пусть $\lambda$ -- разбиение и на $\{x_i\}$ введён порядок $x_1<x_2<\dots <x_i<\dots$ Тогда коэффициент перед мономом $P$ степени $|\lambda|$ в $S_\lambda$ есть количество диаграмм Юнга формы $(\lambda_1, \lambda_2,\dots ,\lambda_k)$ с произведением элементов, равным $P.$

В работе \cite{Ges90} определяются функции $b_i = \sum\limits_{n=0}^{\infty}{x^{2n + i}\over n!(n+i)!}$ и $U_k = \det(b_{\mid i-j\mid})_{1\leqslant i, j\leqslant k}.$

Также вводится функция $R_k(x, y)$ как $R_k(x, y) = \sum\limits_k s_\lambda(x) s_\lambda(y)$, где сумма берётся по всем разбиениям на не более чем $k$ частей. Тогда коэффициент при $x_1 x_2\dots x_n y_1 \dots y_n$ в функции $R_k(x, y)$ равен $\xi_k(n).$ Из этого в \cite{Ges90} выводится, что $U_k = \sum\limits_{n=0}^{\infty}\xi_k(n){ x^{ 2n}\over (n!)^2}.$

Количество полилинейных слов длины $n$ над $l$-буквенным алфавитом $(n \leqslant l)$, в каждом из которых не найдётся последовательности из $(k+1)$ буквы в
порядке лексикографического убывания есть $\bigl(\begin{smallmatrix}l\\ n \end{smallmatrix}\bigr)\xi_k(n)$.

Автор приносит благодарности А.Я. Белову, А.В. Михалёву, В.Н. Латышеву и всем участникам семинара ``Теория колец'' за постановку задачи и обсуждение работы.

\end{document}